\documentclass[3p,times]{elsarticle}
\usepackage{latexsym,amsmath,amssymb,amsthm,amsfonts,stmaryrd,pb-diagram,amscd}
\def\ss{\mathrm{ss}}
\def\RR{\mathbb{R}}
\def\ZZ{\mathbb{Z}}

\DeclareMathOperator\id{\mathrm{id}}

\DeclareMathOperator\ad{\mathrm{ad}}
\DeclareMathOperator\Ad{\mathrm{Ad}}

\DeclareMathOperator\Ht{\mathrm{ht}}
\def\fg{\mathfrak{g}}
\def\fp{\mathfrak{p}}

\def\sP{\mathcal{P}}
\theoremstyle{plain}
\newtheorem{lem}{Lemma}
\newtheorem{pro}{Proposition}
\newtheorem{thm}{Theorem}
\newtheorem*{lem*}{Lemma}
\newtheorem*{pro*}{Proposition}
\newtheorem*{thm*}{Theorem}
\newtheorem*{cor*}{Corollary}

\begin{document}
\begin{frontmatter}

\title{Wei-Norman equations for classical groups \emph{via} cominuscule induction}

\author[jg]{Jan Gutt}
\ead{jan.gutt@gmail.com}

\author[sc]{Szymon Charzy\'nski}
\ead{szycha@cft.edu.pl}

\author[jg]{Marek Ku\'s}
\ead{marek.kus@cft.edu.pl}

\address[jg]{Center for Theoretical Physics, Polish Academy of Sciences, Al. Lotnik\'ow 32/46, 02-668 Warszawa, Poland}

\address[sc]{Chair of Mathematical Methods in Physics, Department of Physics, University of Warsaw, ul. Pasteura 5, 02-093 Warszawa, Poland}


\begin{abstract}
We show how to reduce the nonlinear Wei-Norman equations, expressing the
solution of a linear system of non-autonomous equations on a Lie algebra, to a
hierarchy of matrix Riccati equations using the cominuscule induction. The
construction works for all reductive Lie algebras with no simple factors of
type $G_2$, $F_4$ or $E_8$. A corresponding hierarchy of nonlinear, albeit no
longer Riccati equations, is given for these exceptional cases.

\end{abstract}


\end{frontmatter}

\section{Introduction}

The Wei-Norman method \cite{Wei1963} was developed by its authors to reduce a
systems of linear differential equations with variable coefficients to a
nonlinear one. At first sight advantages of such a reduction seem questionable,
but in fact, in many applications it provides an useful method of analyzing
linear non-autonomous systems, alternative to the commonly used in physics
expansion in a series of time-ordered multiple integrals \cite{Wei1963}.
Another advantage can follow from the fact that the nonlinear system can be
integrated in steps, each involving smaller number of variables than the
original linear system. In \cite{ck13} we were able to show that in the unitary
case, i.e.\ when the original linear system of equations describes an evolution
within the unitary group, the Wei-Norman equations can be reduced to a system
of nonlinear matrix Riccati equations. In \cite{Charzynski2014} we generalized
this result to all classical Lie groups. We also observed that the resulting
equations are in general not of the matrix Riccati form for the non-classical
simple Lie groups showing this fact explicitly for the $G_2$ group, but leaving
other non-classical cases inconclusive. In the present paper we give a unified
approach to the problem based on the cominuscule induction, which works for all
complex reductive Lie groups having no simple factors of type $G_2$, $F_4$ or
$E_8$. Moreover for the cases of simple Lie groups for which the cominuscule
induction is not applicable, we show that employing the \emph{contact grading}
of the corresponding Lie algebra gives the Wei-Norman equations in form of
coupled first-order equations of at most fourth degree (compared to two in the
cominuscule case). Since the contact grading exists for all simple Lie algebras
our results exhibit the structure of the Wei-Norman equations for all reductive
complex Lie groups.

\sloppy
Let $G$ be a complex reductive Lie group, and $\fg$ its Lie algebra. Denote by
$R_g : G \to G$ the map $h \mapsto hg$. Given a continuous map $X : I \to \fg$ from an
interval $I \subset \RR$ containing $0$, the \emph{development of $X$ in $G$} is a
solution $x : I \to G$ of the ODE
$$
\frac{dx}{dt} = R_{x*} X,\qquad x(0)=e_G.
$$
Given a vector space $E$, a \emph{vector Riccati equation} for a function $\xi
: I \to E$ is an ODE of the form
$$
\frac{d\xi(t)}{dt} = \gamma(t) + \alpha(t)(\xi) + \beta(t)(\xi,\xi)
$$
where $\gamma : I \to E$, $\alpha : I \to E^* \otimes E$, $\beta : I \to S^2E^*
\otimes E$ are given functions corresponding to terms of degree, respectively,
zero, one and two in $\xi$.

We shall prove the following.
\begin{thm}\label{thm:main}
Assume no simple factor of $G$ is of type $G_2$, $F_4$ or $E_8$. Then computing
the development of $X$ in $G$ can be, locally on $I$, reduced to solving a
system of finitely many vector Riccati equations on functions $\xi_i : I \to
\fg$ of the form
\[
\frac{d\xi_i(t)}{dt} = \gamma_i(t) + \alpha_i(t)(\xi_i) +
\beta_i(t)(\xi_i,\xi_i),\quad \xi_i(0)=0
\]
for $i=1,\dots,r$, where the coefficients $\gamma_i(t),\alpha_i(t),\beta_i(t)$
can be expressed explicitly in terms of $\xi_j(t)$, $j<i$, and the structure
constants of $\fg$ in a Chevalley basis. The solution is then given in the
Wei-Norman form $x(t) = e^{\xi_1(t)}\cdots e^{\xi_r(t)}$. Furthermore $r\leq
2\, \mathrm{rank}\,G^{ss}+\mathrm{rank}\,G$, where $G^{ss}$ is the semisimple
part of $G$.
\end{thm}
We stress that the method is intrinsic to the Lie algebra $\fg$, and does not rely on
any particular representation (i.e.\ does not proceed by reduction to the classical
result for matrix equations). Our approach uses the existence of so-called cominuscule
parabolic subgroups, which appear precisely in those reductive groups that satisfy the
hypothesis of the Theorem.

Let us briefly recall the notion of a general parabolic subgroup of a reductive complex
Lie group $G$: it is a proper closed subgroup $P\subset G$ such that the homogeneous
space $G/P$ is a compact complex manifold (in fact, a projective algebraic variety). The
group $P$ admits a Levi decomposition $P = L \ltimes U$ where $U$ is the unipotent
radical of $P$, while $L\simeq P/U$ is a reductive group called the \emph{Levi factor}
of $P$. On the level of Lie algebras, one may associate with $P$ a grading
$$
\fg=\fg_{-k}\oplus\cdots\oplus\fg_{-1}\oplus\fg_0\oplus\fg_1\oplus\cdots\oplus\fg_k,
\quad [\fg_i,\fg_j]\subset \fg_{i+j}
$$
for certain $k > 0$, such that $\fg_0\oplus\cdots\oplus\fg_k$ is the Lie algebra of $P$,
and $\fg_0$ is the Lie algebra of $L$ (it is understood that $\fg_i = 0$ whenever $|i|
> k$). In fact, the grading as above determines the subgroup $P$ up to isogeny,
and a maximal compatible $P$ may be chosen by requiring that $G/P$ is connected and
simply-connected.

Associated with $P$ is another parabolic subgroup $P^\prime$, called the \emph{opposite
parabolic}, with the property that $P\cap P^\prime = L$ and $P^\prime\cdot P$ is dense
in $G$. Writing $P^\prime = U^\prime \rtimes L$ for the Levi decomposition of the
opposite parabolic, one has that the map
$$
U^\prime\times L\times U \rightarrow G
$$
induced by the group operation is an embedding onto a dense open subset. The grading
corresponding to $P^\prime$ is simply the opposite of the grading associated with $P$,
so that in particular $\fg_{-k}\oplus\cdots\oplus\fg_{-1}$ is the Lie algebra of
$U^\prime$, and the above open embedding gives
$$
\overset{\mathrm{Lie}\,U^\prime}{\overbrace{\fg_{-k}\oplus\cdots\oplus\fg_{-1}}}\oplus
\overset{\mathrm{Lie}\,L}{\overbrace{\fg_0}}\oplus
\overset{\mathrm{Lie}\,U}{\overbrace{\fg_1\oplus\cdots\oplus\fg_k}}
$$
on the infinitesimal level.

\section{Proof of the main Theorem}

A  parabolic subgroup of a reductive complex Lie group is called \emph{cominuscule} if
it satisfies the equivalent statements in the following:
\begin{pro}[cf. \cite{rrs,cs}]\label{pro:comin}
Let $P \subset G$ be a  parabolic subgroup. The following are equivalent:
\begin{enumerate}
\item $G/P$ is a Hermitian symmetric space of compact type,
\item the unipotent radical of $P$ is abelian,
\item there is a grading $\fg = \fg_{-1} \oplus \fg_0 \oplus \fg_1$ such
    that $\fp = \fg_0 \oplus \fg_1$.
\end{enumerate}
\end{pro}
The proof of the Theorem~\ref{thm:main} relies on \emph{cominuscule induction} (see
e.g.\ \cite{hwang-mok-rigidity} for a geometric application). The following will be
proven later.
\begin{pro}
\label{pro:ind} Let $\sP$ be a property of complex reductive Lie groups such
that:
\begin{enumerate}
\item $\sP$ holds for all tori,
\item if $\sP$ holds for the Levi factor of some cominuscule parabolic
    subgroup in $G$, then it holds for $G$.
\end{enumerate}
Then $\sP$ holds for all reductive complex Lie groups with no simple factors of
type $G_2$, $F_4$ or $E_8$.
\end{pro}

\begin{proof}[Proof of Theorem \ref{thm:main}]
Given a complex reductive $G$ we let $\sP(G)$ stand for the conclusion of the Theorem.
Clearly, computing the development in a torus reduces to integrating and exponentiating
a function, hence to a degenerate vector Riccati equation, so that $\sP$ holds for all
tori.

Let now $P \subset G$ be a cominuscule  parabolic subgroup. We use the grading
$$ \fg = \fg_{-1} \oplus \fg_0 \oplus \fg_1 $$
such that $\fp = \fg_0\oplus\fg_1$ (cf. Prop.~\ref{pro:comin}). One then has that
$\fg_0$ is the Lie algebra of the Levi factor $L$ of $P$, while $\fg_1$ is the Lie
algebra of its unipotent radical $U$. The opposite parabolic $P'$ has the property that
$P' \cap P = L$ and $L$ is the Levi factor of $P'$. One then has $\fp = \fg_{-1} \oplus
\fg_0$ and $\fg_{-1}$ is the Lie algebra of the unipotent radical $U'$ of $P'$. The
group operation induces an embedding
$$ U' \times L \times U \hookrightarrow G $$
onto a dense open subset. The exponential map gives Lie group
\emph{isomorphisms} $\fg_{-1} \simeq U'$, $\fg_1 \simeq U$, where the Lie group
structure on $\fg_{\pm1}$ is induced by the vector space structure. Observe
that $\ad_X$ is a derivation of degree $\pm1$ for $X \in \fg_{\pm1}$, while
$\fg$ is graded in degrees $-1,0,1$, so that $\ad_X^3=0$. It follows that
$$
\Ad_{\exp X} = \sum_k \frac{1}{k!} \ad_X^k =
\id + \ad_X + \frac{1}{2}\ad_X^2,\qquad X \in \fg_{\pm1}.
$$
We shall use the exponential to view elements of $\fg_{\pm1}$ as elements of
$G$, so that
$$
\Ad_XY = Y + [X,Y] + \frac{1}{2}[X,[X,Y]],\qquad X \in \fg_{\pm1},\ Y \in \fg.
$$

Consider now the development equation
$$
\frac{dx}{dt} = R_{x*}X,\quad x(0)=e_G.
$$
Possibly shrinking the domain interval $I$, we can assume $x(t) \in U'LU$, so
that the equation becomes equivalent to
$$
R_{WyV*}^{-1} \frac{d}{dt} WyV = Z'+ Y+ Z,\quad W(0)=0,\ y(0)=e_L,\ V(0)=0
$$
where $$X = Z'+Y+Z,\quad x = WyV$$ for $$Z',W : I \to \fg_{-1},\quad Z,V : I
\to \fg_1$$ and $$Y : I \to \fg_0,\quad y: I \to L.$$ We compute:
$$
R_{WyV*}^{-1} \frac{d}{dt} WyV =
\frac{dW}{dt} + \Ad_W R_{y*}^{-1}\frac{dy}{dt}
+\Ad_W \Ad_y \frac{dV}{dt}.
$$
Comparing with $Z'+Y+Z$ and decomposing according to the grading, we then  have
\begin{eqnarray*}
\frac{dW}{dt} + [W, R_{y*}^{-1}\frac{dy}{dt}] + \frac{1}{2}
[W,[W, \Ad_y\frac{dV}{dt}]] &=& Z' \\
R_{y*}^{-1} \frac{dy}{dt} + [W, \Ad_y\frac{dV}{dt}] &=& Y \\
\Ad_y \frac{dV}{dt} &=& Z.
\end{eqnarray*}
Elimination yields the equivalent system
\begin{eqnarray*}
\frac{dW}{dt} + [W,Y] - \frac{1}{2} [W,[W,Z]] &=& Z' \\
R_{y*}^{-1} \frac{dy}{dt} + [W, Z] &=& Y \\
\Ad_y \frac{dV}{dt} &=& Z.
\end{eqnarray*}
It follows that solving the development equation is equivalent to the following
sequence:
\begin{enumerate}
\item solving the Riccati equation
$$ \frac{dW}{dt} = Z' - [W,Y] + \frac{1}{2} [W,[W,Z]] $$
on $W : I \to \fg_{-1}$,
\item solving the development equation
$$ \frac{dy}{dt} = R_{y*} (Y - [W,Z]) $$
on $y : I \to L$,
\item integrating
$$ V = \int \Ad_y^{-1}Z dt. $$
\end{enumerate}
Clearly, steps (2) and (3) can be considered as degenerate Riccati equations
themselves. Denoting the conclusion of the Theorem with $\sP(G)$, we have just
proven that $\sP$ satisfies the hypotheses of Prop. \ref{pro:ind}. Hence it
holds for all reductive complex Lie groups whose simple factors admit
cominuscule parabolic subgroups.
\end{proof}

\section{Proof of Parabolic Induction}
It remains to prove Proposition \ref{pro:ind}. We recall the description of
parabolic subgroups in terms of root data:
\begin{pro}[cf.~\cite{cs}]
Let $G \supset B \supset T$ be a complex reductive Lie group, a Borel subgroup, and a
maximal torus. Denote by $\Phi$ the root system of $\fg$ with respect to $T$, and by
$\Phi^+$ the subset of positive roots determined by $B$ (i.e.\ $B$ is generated by $T$
and by root subgroups corresponding to positive roots). Let $\Delta \subset \Phi^+$ be
the simple roots. Define for each subset $\Sigma \subset \Delta$ the function
$\Ht_\Sigma : \Phi \to \ZZ$ sending a root $\alpha\in\Phi$ to the sum of the
coefficients of elements of $\Sigma$ in the simple root decomposition of $\alpha$. Let
$P_\Sigma \subset G$ be the subgroup generated by $B$ and root subgroups $G_{-\alpha}$
with $\Ht_\Sigma(\alpha)=0$. Then:\begin{enumerate}
\item $P_\Sigma$ is parabolic,
\item $\Sigma \mapsto P_\Sigma$ gives a one-to-one correspondence between
    nonempty subsets of $\Delta$ and conjugacy classes of parabolic subgroups
    in $G$,
\item the Dynkin diagram of the Levi factor $L_\Sigma$ of $P_\Sigma$ is
    obtained by removing nodes corresponding to $\Sigma \subset \Delta$ from
    the Dynkin diagram of $G$.
\end{enumerate}
\end{pro}
Parabolic subgroups of the form $P_\Sigma$ are called standard; every parabolic subgroup
is standard for some choice of a maximal torus and positive roots. In particular,
cominuscule parabolic subgroups are characterized as follows:
\begin{lem}[cf.~\cite{rrs,cs}]
Let $P_\Sigma$ be a standard parabolic for $\Sigma \subset \Delta$. The
following are equivalent:
\begin{enumerate}
\item $P_\Sigma$ is cominuscule,
\item $\Ht_\Sigma(\Phi_+) = \{0,1\}$,
\item marking the nodes corresponding to $\Sigma$ in the Dynkin diagram
    $\Gamma$ of $G$, each connected component of $\Gamma$ contains either
    no marked nodes, or a single marked cominuscule node, where the
    cominuscule nodes for connected diagrams are as follows (in Bourbaki
    labelling):
\begin{eqnarray*}
A_n &: & \alpha_i,\ 1 \le i \le n \\
B_n &: & \alpha_1 \\
C_n &: & \alpha_n \\
D_n &: & \alpha_1, \alpha_{n-1}, \alpha_n \\
E_6 &: & \alpha_1, \alpha_6 \\
E_7 &: & \alpha_7.
\end{eqnarray*}
\end{enumerate}
\end{lem}
As a consequence, we have:
\begin{lem}\label{lem:ind}
Assume no simple factor of $G$ is of type $G_2$, $F_4$, $E_8$. Then $G$ admits
a cominuscule parabolic subgroup, and furthermore for each cominuscule
parabolic subgruop $P\subset G$, the Levi factor $L$ of $P$ has no simple
factor of type $G_2$, $F_4$, $E_8$.
\end{lem}
\begin{proof}By inspection of Dynkin diagrams.\end{proof}

\begin{proof}[Proof of Proposition \ref{pro:ind}]
We shall proceed by induction on the semisimple rank. Note that the Proposition
holds tautologically when restricted to tori. Suppose the Proposition holds
when restricted to reductive complex Lie groups of semisimple rank less than
$n$. Let $\sP$ be a property of complex reductive Lie groups satisfying the
hypotheses of the Proposition, and let $G$ be a reductive complex Lie group of
semisimple rank $n$, and with no simple factors of type $G_2$, $F_4$ or $E_8$.
By Lemma \ref{lem:ind}, there is a cominuscule parabolic subgroup $P \subset
G$, and $\sP$ holds for the Levi factor $L$ of $P$ by the inductive hypothesis
($\dim L^\ss < n$). Hence $\sP$ holds for $G$ by hypothesis (2) of the
Proposition.
\end{proof}

\section{Remark on the remaining cases of simple groups}

As the reader of our proof will notice, it is not difficult to give a weaker general
result for all reductive groups. The reduction to Riccati equations relied on the single
fact that $\mathrm{ad}^3_X = 0$ for $X\in\fg_{\pm 1}$ in the cominuscule grading; the
class of groups under consideration was precisely those whose Lie algebras admit such
gradings. It turns out that the next best thing happens in the general case: every
simple Lie algebra of rank at least two admits a contact grading, i.e.\ one of the form
$$
\fg =\fg_{-2}\oplus\fg_{-1}\oplus\fg_0\oplus\fg_1\oplus\fg_2
$$
with $\mathrm{dim}\fg_{\pm 2} = 1$ (cf. \cite{cs}). We then have $\mathrm{ad}^5_X = 0$
for $X\in\fg_{\pm 1}$ and $\mathrm{ad}^4_X = 0$ for $X\in\fg_{\pm 2}$, whence following
the lines of the proof of Theorem 1 for such a grading one is left with a system of
first-order ODEs of degree at most four (compared to two in the cominuscule case).

\section*{Acknowledgements}

The presented results are obtained in frames of the the Polish National Science Center
project MAESTRO DEC-2011/02/A/ST1/00208 support of which is gratefully acknowledged by
all authors.


\end{document}